\documentclass[11pt]{article}
\usepackage[textwidth=160mm,textheight=200mm]{geometry}
\usepackage{amsmath}
\usepackage{amssymb}
\usepackage{amsthm}
\usepackage{graphicx}
\usepackage{microtype}
\usepackage{hyperref}

\newtheorem{theorem}{Theorem}[section]
\newtheorem{lemma}[theorem]{Lemma}

\newtheorem{proposition}[theorem]{Proposition}
\newtheorem{corollary}[theorem]{Corollary}
\theoremstyle{definition}

\newtheorem{definition}[theorem]{Definition}

\newtheorem{example}[theorem]{Example}

\newtheorem{remark}[theorem]{Remark}

\def\A{\mathcal{L}}

\def\3{2}

\begin{document}
\title{\textbf{Pseudo--Anosov mapping classes from pure mapping classes}}
\author{Yohsuke Watanabe}
\maketitle
\begin{abstract}
We study types of mapping classes which arise as a product of a given mapping class and powers of certain pure mapping classes. We derive an explicit constant depending only on a surface such that almost all above pure mapping classes give rise to pseudo--Anosov type whenever their powers are larger than the constant. Furthermore, the stable lengths of pseudo--Anosov mapping classes obtained by this method are directly captured from the construction. 
\end{abstract}



\section{Introduction}\label{introduction}
Let $S=S_{g,n}$ be a genus $g$ surface with $n$ punctures. We let $\xi(S_{g,n}):=3g+n-3$ and $\chi(S_{g,n}):=2-2g-n$ denote the complexity and the Euler characteristic of $S$ respectively. In this paper, we focus on $S$ such that $\xi (S)>1$. 
The curve graph $C(S)$ is defined as follows: the vertices are isotopy classes of curves and the edge between two vertices are realized by disjointness. 
We let $d_{S}$ denote the graph metric on $C(S)$ and if $A,B\subseteq C(S)$ then $$d_{S}(A,B):=\max\{d_{S}(a,b)|a \in A,b\in B \}.$$ 
The mapping class group $Mod(S)$ is the group of orientation preserving self--homeomorphisms of $S$ up to isotopy. Nielsen--Thurston classification says that $f\in Mod(S)$ is periodic, reducible or pseudo--Anosov. See \cite{FM} for the definitions and discussions.  
$Mod(S)$ acts naturally on $C(S)$ by an isometry. Let $f\in Mod(S)$, we let $||f||$ denote the stable length of $f$, which is defined by $$||f||:=\liminf_{j\rightarrow \infty}\frac{d_{S}(x,f^{j}(x))}{j}$$ where $x\in C(S)$. If $f$ is periodic or reducible, some vertex of $C(S)$ is fixed by some power of $f$, so $||f||=0$. Masur--Minsky showed that if $f$ is pseudo--Anosov then $||f||>0$ \cite{MM1}. The stable lengths of pseudo--Anosov mapping classes have been studied actively, for instance see the work of Farb--Leininger--Margalit \cite{FarbLeiningerMargalit}, Gadre--Tsai \cite{GadreTsai}, and Valdivia \cite{Valdivia}. We will need a result of Gadre--Tsai: 
\begin{theorem}[\cite{GadreTsai}]\label{GD}
There exists $k(S)\geq \frac{1}{162 \cdot |\chi(S)|^{2}+30 \cdot |\chi(S)|-10n}$ such that for any pseudo--Anosov mapping class $f\in Mod(S)$, any $x\in C(S)$, and any $j \in \mathbb{Z}$, $$d_{S}(x,f^{j}(x))\geq k(S) \cdot |j|.$$ \end{theorem}

Pseudo--Anosov mapping classes are difficult to obtain directly from their definition. The aim of this paper is to provide a method to generate them by using known technologies developed in the study of mapping class groups, the curve graphs, and the interplay between them via group actions. 
\subsection{Related results and questions}\label{historyupto2016}
We review some known constructions of pseudo--Anosov mapping classes. Recall $A\subseteq C(S)$ is said to fill $S$ if every complementary component of $A$ in $S$ is a disk or a peripheral annulus. Note that $a,b\in C(S)$ fill $S$ if and only if $d_{S}(a,b)>2$.
\begin{itemize}
\item In \cite{THU}, Thurston showed that if $a,b\in C(S)$ such that they together fill $S$, then $t_{a}^{p}\circ t_{b}^{-q}$ is pseudo--Anosov for all $p,q\in \mathbb{N}_{>0}$. This result was generalized to the cases where more than two Dehn twists are involved by Long \cite{longmulti} and then by Penner \cite{Penner}.

\item In \cite{LM}, Long--Morton showed that if $f\in Mod(S)$ is pseudo--Anosov and $t_{a}$ is Dehn twist along $a\in C(S)$ then $t_{a}^{j}\circ f$ is pseudo--Anosov except for finitely many values of $j$. 

\item In \cite{Fathi}, Fathi effectivised the above result of Long--Mortan. He showed that if $f\in Mod(S)$ and $a\in C(S)$ such that $\{f^{n}(a)|  n \in \mathbb{Z}\}$ fills $S$, then $\{t_{a}^{j} \circ f | j\in \mathbb{Z}\}$ are pseudo--Anosov except for at most 7 consecutive values of $j$.

Note that Fathi's theorem generalizes Thurston's theorem because if $a,b\in C(S)$ are filling curves then $\{t_{b}^{n}(a)| n\in \mathbb{Z}\}$ fills $S$, and Fathi's theorem applies. Fathi's 7 was improved to be 6 by Boyer--Gordon--Zhang in \cite{BGZ}.


\item In \cite{POFI}, Ivanov asked the following question, which is a Long's and Penner's like extension on Fathi's theorem: is there a constant $N_{S}$, depending only on $S$, such that the following holds? Let $f\in Mod(S)$ and let $m=t_{a_{1}}^{\pm 1}\circ t_{a_{2}}^{\pm 1} \circ \cdots \circ t_{a_{k}}^{\pm 1}$ be a multitwist. If $\{f^{n}(a_{i})|1\leq i\leq k,n\in \mathbb{Z}\}$ fills $S$, then $\{m^{j} \circ f | j\in \mathbb{Z}\}$ are pseudo--Anosov except for at most $N_{S}$ consecutive values of $j$.    
\end{itemize}

\subsection{Results}\label{alowei}
In this paper, we also study the construction of pseudo--Anosov mapping classes obtained by a product of two mapping classes. Our construction is closer to that of Long--Morton and Fathi: roughly speaking, we will generate pseudo--Anosov mapping classes by multiplying pure mapping classes to a given mapping class. We recall the definition of pure mapping classes; first, recall that if $f$ is reducible then there exists $i$ such that $f^{i}$ behaves as a power of Dehn twist in each component of the regular neighborhoods of some multicurve and behaves as a pseudo--Anosov or the identity map in each complementary component of the neighborhoods. Pure mapping classes are the mapping classes that admit the above decomposition when $i=1$. They consist of the identity map, pseudo--Anosov maps, partial pseudo--Anosov maps, Dehn twists, and multitwists. Note that the mapping class groups contain ``many'' pure mapping classes; Ivanov showed that the level $p$ congruence subgroup, which is a ``finite index'' subgroup of $Mod(S)$, consists of pure mapping classes if $p>2$ \cite{Ivanov}. We let $PMod(S)$ denote the set of all pure mapping classes of $Mod(S).$

We define the following which is necessary to state the main theorem thereafter. (We refer the reader to $\S \ref{shimokita}$ for notations used in Definition \ref{LATER}.)
\begin{definition}\label{LATER}
\begin{itemize}
\item Let $a \in C(S)$. We define $\mathcal{Z}_{a}:=\{Z\subseteq S| \pi_{Z}(a)=\emptyset\}$. Note that $\mathcal{Z}_{a}$ consists of the annulus whose core curve is $a$ and all subsurfaces contained in the complement of $a$.
\item Let $\phi \in PMod(S)$. We call $Z\subseteq S$ a loxodromic domain of $\phi$ if $\phi(Z)=Z$ and there exists $k>0$ such that for any $x\in C(S)$ such that $\pi_{Z}(x)\neq \emptyset$ and any $j\in \mathbb{Z}$, $$d_{Z}(x,\phi^{j}(x))\geq k\cdot |j|.$$ Furthermore, we let $\A(\phi)$ denote the set of all loxodromic domains of $\phi.$ For example,
\begin{itemize}
\item If $f$ is a partial pseudo--Anosov map supported on $Z\subseteq S$ then $\A(f)= \{Z\}$, see Lemma \ref{min}. 
\item Let $a \in C(S)$. If $t_{a}$ is Dehn twist along $a$ then $\A(t_{a}) = \{\text{The annulus whose core curve is }a \}$ and if $m_{a}$ is a multitwist which contains $t_{a}$ then $\A(m_{a})\ni \{\text{The annulus whose core curve is }a \}.$ For both cases, $k$ can be taken to be $1$, see Lemma \ref{dehnmin} and Lemma \ref{multimin}.
\end{itemize}
\item Let $a \in C(S)$. We define $\Phi_{a}:=\{\phi \in PMod(S)| \phi(a)=a \text{ and } \A(\phi)\neq \emptyset\}$. Note that 
\begin{itemize}
\item $\Phi_{a}$ consists of non--identity pure mapping classes which fix $a$.
\item If $\phi \in \Phi_{a}$ then $\A(\phi)\subseteq \mathcal{Z}_{a}$.
\end{itemize}
\end{itemize}
\end{definition}

Now, we state the main theorem of the paper: throughout this section, we let $M\leq 100$ and $k(S)\geq \frac{1}{162 \cdot |\chi(S)|^{2}+30 \cdot |\chi(S)|-10n}$ denote the constants given by Theorem \ref{BGIT} and by Corollary \ref{gru} respectively. 
\begin{theorem}\label{almostall}
Let $f\in Mod(S)$ and $a \in C(S)$ such that $d_{S}(a,f(a))>4$. There exists $\mathcal{W} \subseteq \mathcal{Z}_{a}$ with $|\mathcal{W}|<\infty$ such that the following holds: if $\phi \in \Phi_{a}$ such that $\A(\phi) \cap \{\mathcal{Z}_{a} \setminus \mathcal{W} \}\neq \emptyset$, then $\phi^{j}\circ  f$ is pseudo--Anosov whenever $|j|> \frac{M+14}{k(S)}$.
\end{theorem}
We state some remarks regarding the above theorem:
\begin{remark}\label{aboutthemt}
First, we state some comparison with the results of Long--Morton and Fathi from $\S\ref{historyupto2016}$:
even though a given mapping class, denoted $f$ in Theorem \ref{almostall}, does not have to be pseudo--Anosov, the hypothesis of $f$ is stronger than that of Fathi. Indeed, Fathi only requires that $d_{S}(a,f^{n}(a))>2$ for some $n$ so that $\{f^{n}(a)|  n \in \mathbb{Z}\}$ fills $S$.
While Long--Morton and Fathi assume a surface, denoted $S$ in Theorem \ref{almostall}, to be closed, our proof applies to a general surface.
Unlike with Fathi's uniform constant, 7, our constant depends on a surface, but our method captures the position of the mapping classes that are not pseudo--Anosov when we vary $j$ in Theorem \ref{longmorton}, Theorem \ref{fathi} and Theorem \ref{almostall} over $\mathbb{Z}$. 
While Long--Morton and Fathi focus on a specific pure mapping class, $t_{a}$, we consider all non--identity pure mapping classes which fix $a$, $\Phi_{a}$. Clearly, $t_{a}\in \Phi_{a}$. Theorem \ref{almostall} says that almost all elements of $\Phi_{a}$ give rise to pseudo--Anosov type as far as they are raised to the $j$th power where $|j|> \frac{M+14}{k(S)}$. Capturing this phenomena for a large family of $\Phi_{a}$ is the main contribution of this paper. However, we can still apply Theorem \ref{almostall} to specific families of $\Phi_{a}$; for example, Dehn twists and multitwists. Since $|\mathcal{W}|<\infty$, there are only finitely many Dehn twists and multitwists in $\Phi_{a}$ such that the statement of Theorem \ref{almostall} does not hold for. 
Lastly, we note that exceptional elements of $\Phi_{a}$, which the statement of Theorem \ref{almostall} does not hold for, are detectable since $\mathcal{W}$ can be explicit by the proof of Theorem \ref{almostall}; $\mathcal{W}=\{W\subseteq \mathcal{Z}_{a}| d_{W}(f^{-1}(a),f(a))>3\}.$ We can further analyze the elements of $\mathcal{W}$; by the proof of Theorem \ref{BBF}, which we refer the reader to \cite{BBF}, if $W\in \mathcal{W}$ then we must have $i(f^{}(a), \partial(W)) \leq 2\cdot i(f^{-1}(a), f(a))$ and $i(f^{-1}(a), \partial(W))\leq 2\cdot i(f^{-1}(a), f(a)).$ 
\end{remark}

The proof of the main theorem given in $\S\ref{tl2}$ gives rise to the results which follow.

Since the description of $\mathcal{W}$ in Theorem \ref{almostall} is explicit as in Remark \ref{aboutthemt}, we can pick a pure mapping class in $\Phi_{a}$ so that the statement of Theorem \ref{almostall} holds for it:
\begin{corollary}\label{ccc1}
Let $f\in Mod(S)$ and $a \in C(S)$ such that $d_{S}(a,f(a))>4$. If $\phi \in \Phi_{a}$ such that $i(f^{}(a), \partial(Y)) >2\cdot i(f^{-1}(a), f(a))$ or $i(f^{-1}(a), \partial(Y))>2\cdot i(f^{-1}(a), f(a))$ for some $Y\in \A(\phi)$, then $\phi^{j}\circ  f$ is pseudo--Anosov whenever $|j|> \frac{M+14}{k(S)}$.
\end{corollary}

A pure mapping class $\phi$ in Corollary \ref{ccc1} can be found easily. See Remark \ref{dry1}.

We can add an extra assumption to Theorem \ref{almostall} so that the statement of Theorem \ref{almostall} holds for every element of $\Phi_{a}$:
\begin{corollary}\label{ccc2}
Let $f\in Mod(S)$ and $a \in C(S)$ such that $d_{S}(a,f(a))>4$ and that there exists a geodesic between $f^{-1}(a)$ and $f(a)$ whose $1$--neighborhood does not intersect with $\{\partial(Z)|Z\in \mathcal{Z}_{a}\}$ in $C(S)$. If $\phi \in \Phi_{a}$, then $\phi^{j}\circ  f$ is pseudo--Anosov whenever $|j|> \frac{2\cdot M+11}{k(S)}.$ 
\end{corollary}
Note that the diameter of $\{\partial(Z)|Z\in \mathcal{Z}_{a}\}$ is at most $2$ in $C(S)$.
A pair $f\in Mod(S)$ and $a \in C(S)$ in Corollary \ref{ccc2} can be found easily.  See Remark \ref{dry2}.

\begin{remark}
Pseudo--Anosov mapping classes obtained in this paper are special from the view point of their actions to the curve graph: their stable lengths are directly captured from the construction. In particular, if $g$ is a pseudo--Anosov mapping class obtained by Theorem \ref{almostall}, Corollary \ref{ccc1}, and Corollary \ref{ccc2}, then we have $d_{S}(a,f(a))-4 \leq ||g||\leq d_{S}(a,f(a))$. See Remark \ref{fix2}. Indeed, in $\S\ref{tl1}$, we construct pseudo--Anosov mapping classes which have invariant geodesics in the curve graph where we have $||g||= d_{S}(a,f(a))$. See Remark \ref{fix1}. 
\end{remark}

\subsection{Plan of the paper} In $\S\ref{preok}$, we collect necessary ingredients to prove the main theorem. First we will review the proofs given by Long--Morton ($\S \ref{sectionlm}$) and Fathi ($\S\ref{sectionf}$) in their constructions. Their approaches work beautifully with Dehn twists and multitwists, but not with partial pseudo--Anosov mapping classes. 
Our approach relies on an elementary criteria from $Mod(S) \curvearrowright C(S)$: if $f\in Mod(S)$ and $||f||>0$, then $f$ is pseudo--Anosov. Therefore, we will need to be able to capture the orbits of the mapping classes in $C(S)$, and we will use subsurface projection theory for this. Once again, the main difference with the constructions of Long--Morton and Fathi is that our main construction is not designed to produce pseudo--Anosov mapping classes with a specific type of pure mapping classes. In fact, the benefit of working with subsurface projection theory is that we can treat all types of pure mapping classes as the ``same'' type in virtue of their loxodromic actions on subsurfaces, see $\S\ref{machida}$. 

In $\S\ref{tl}$, we prove the main theorem. $\S\ref{tl1}$ presents a weaker/specific version of the main theorem with a weaker hypothesis on a given map $f$. The reader should think of $\S\ref{tl1}$ as the warm--up of $\S\ref{tl2}$. In $\S\ref{tl2}$, we will first obtain a technical proposition (Proposition \ref{i3}) which captures the orbits of mapping classes; the proof is a simple application of Theorem \ref{BGIT} and Theorem \ref{BEH}. Then we consider a family of mapping classes obtained by multiplying certain pure mapping classes to a given mapping class (Lemma \ref{actual}), and we show that, regardless of the type of pure mapping classes, almost all pure mapping classes give rise to pseudo--Anosov mapping classes (Theorem \ref{whatweprove}). We emphasize that the implication from Lemma \ref{actual} to Theorem \ref{whatweprove} is possible in virtue of Theorem \ref{BBF}.
The above three key tools (Theorem \ref{BGIT} and Theorem \ref{BBF} by Masur--Minksy, and Theorem \ref{BEH} by Behstock) will be reviewed in $\S \ref{3important}$. In fact, we will state and use the effective versions of them (Theorem \ref{BGIT} by Webb, Theorem \ref{BEH} by Leininger, and Theorem \ref{BBF} by Bestvina--Bromberg--Fujiwara). The explicit constants given by these effective results allow us to obtain an explicit constant in the main theorem. Corollary \ref{ccc1} and Corollary \ref{ccc2} come naturally out of the arguments given to prove the main theorem.


\renewcommand{\abstractname}{\textbf{Acknowledgements}}
\begin{abstract}
The author thanks Kenneth Bromberg for comments on this paper, Erik Guentner for inspirational conversations, Asaf Hadari, Dan Margalit, and Bal\'azs Stenner for useful conversations, and Tarik Aougab, Johanna Mangahas, and Samuel Taylor for explaining a work of Gadre--Tsai. 
Some parts of this paper were conisdered while the author was participating the following conferences held in 2016: \textit{Effective and Algorithmic Methods in Hyperbolic Geometry} at ICERM, \textit{Advanced School on Geometric Group Theory and Low-Dimensional Topology: Recent Connections and Advances} at ICTP, \textit{The 63rd Topology Symposium} at Kobe University, and \textit{Geometry of mapping class groups and $Out(F_n)$} at MSRI. The author thanks all organizers for supports on trips. Finally, the author thanks his home institute, the University of Hawaii at Manoa, for supports on trips. 
\end{abstract}

\section{Background}\label{preok}
First, we will review the approaches taken by Long--Morton \cite{LM} and Fathi \cite{Fathi}, and discuss some difficulties in applying their approaches to the setting of general pure mapping classes. 
\subsection{Long--Morton's approach}\label{sectionlm}
Long--Morton showed
\begin{theorem}[\cite{LM}]\label{longmorton}
Suppose $S$ is closed. Let $f\in Mod(S)$ and $a\in C(S)$. If $f$ is pseudo--Anosov, then $\{t_{a}^{j}\circ f | j\in \mathbb{Z}\}$ are pseudo--Anosov except for finitely many values of $j$
\end{theorem}

Long--Morton use the following criterion due to Thurston: $f$ is pseudo--Anosov if and only if the mapping torus of $f$, denoted $M(f)$, is hyperbolic. Furthermore, they use the following fact due to Harer \cite{HARER} and Stallings \cite{STALLINGS}: $M(t_{a}^{j}\circ f )$ can be obtained from $M(f)$ by doing $(1,j)$--Dehn filling on $M(f)\setminus N_{\epsilon}(a)$ relative to a suitable choice of longitude in $\partial(N_{\epsilon}(a))$. They showed that $M(f)\setminus N_{\epsilon}(a)$ is hyperbolic \cite[Lemma 1.1 \& Theorem 1.2]{LM}, then the statement directly follows by Thurston's hyperbolic Dehn surgery theorem. Note that there have been many studies done toward understanding the maximal number of exceptional slopes in the Dehn surgery theorem. For instance, see the recent works by Agol \cite{Agol} and Lackenby--Meyerhoff \cite{LackenbyMeyerhoff} and references therein.

Lastly, we remark that if $t_{a}$ is replaced by a partial pseudo--Anosov mapping class $\phi$, then it is not straightforward to find a hyperbolic manifold $W$ so that Dehn fillings on $W$ give rise to $M(\phi^{j}\circ f )$.  

\subsection{Fathi's approach}\label{sectionf}
Fathi showed
\begin{theorem}[\cite{Fathi}]\label{fathi}
Suppose $S$ is closed. Let $f\in Mod(S)$ and $a\in C(S)$. If $\{f^{n}(a)|  n \in \mathbb{Z}\}$ fills $S$, then $\{t_{a}^{j}\circ f | j\in \mathbb{Z}\}$ are pseudo--Anosov except for at most 7 consecutive values of $j$.
\end{theorem}

Fathi uses the following criterion due to Thurston: if $f$ acts on the space of measured foliations of $S$, denoted $\mathcal{MF}(S)$, without finite orbits, then $f$ is pseudo--Anosov. Fathi further reduced the criterion to the following \cite[Lemma 1.2]{Fathi}: ($\star$): Let $C>1$. If, for all $\mathcal{F}\in \mathcal{MF}(S)$, there exists $p,q\in \mathbb{Z}$ with $(\star 1)$: $i(a, h^{p}(\mathcal{F}))\geq C\cdot i(a, \mathcal{F})$ and $(\star 2)$: $i(a, h^{q}(\mathcal{F}))>0$ then $h$ does not act on $\mathcal{MF}(S)$ with finite orbits, hence $h$ is pseudo--Anosov.
Main ingredient \cite[Proposition 5.2]{Fathi}, which essentially follows from \cite[Theorem 4.1 \& Theorem 4.4]{Fathi}, is the following: there exists a constant $\lambda_{0} $ such that for all $j\in \mathbb{Z}$ and $\mathcal{F}\in \mathcal{MF}(S)$, $$\max \bigg\{i(a, [t_{a}^{j}\circ f]^{k}(a)), i(a, [t_{a}^{j}\circ f]^{-k}(a)) \bigg \}\geq \frac{|j-\lambda_{0}|-1}{2} \cdot i(a, f^{k}(a))\cdot i(a, \mathcal{F}).$$ 
Now, we investigate how $t_{a}^{j}\circ f$ can be pseudo--Anosov by using the above Fathi's reduced criterion ($\star$). First, since $\{f^{n}(a)| n\in \mathbb{Z}\}$ fills $S$, $\{[t_{a}^{j}\circ f]^{n}(a)| n\in \mathbb{Z}\}$ also fills $S$ for all $j \in \mathbb{Z}$ \cite[Lemma 5.3]{Fathi}, so we have $(\star 2)$.
For $(\star 1)$, let $C=\frac{|j-\lambda_{0}|-1}{2} \cdot i(a, f^{k}(a)).$ By ($\star$), if $C>1$ i.e. if $|j-\lambda_{0}|>1+ \frac{2}{i(a, f^{k}(a))}$, then $t_{a}^{j}\circ f$ is pseudo--Anosov. In other words, if neither $t_{a}^{n}\circ f$ nor $t_{a}^{m}\circ f$ is pseudo--Anosov, then $$|n-m|\leq 2+ \frac{4}{i(a, f^{k}(a))}\leq 6$$ and that is how Fathi's number $7$ comes from.

Lastly, we remark that Fathi's argument heavily relies on \cite[Theorem 4.1]{Fathi}, which makes use of Dehn twists and multitwists, see \cite[$\S4$]{Fathi}. His method does not seem to apply in the setting of a partial pseudo--Anosov mapping class $\phi$.

\subsection{Our approach}\label{sectionw}
Our approach uses an obvious criterion: if $f\in Mod(S)$ and $||f||>0$, then $f$ is pseudo--Anosov. Main machinery we use here is subsurface projection theory, much of which are due to Masur--Minsky \cite{MM2} and Behrstock \cite{BEH}. Benefit in working with subsurface projection theory is that we do not need to distinguish the types of pure mapping classes that we deal with.

\subsubsection{Subsurface projections}\label{shimokita}
The arc and curve graph $AC(S)$ is defined as follows; the vertices are isotopy classes of arcs and curves, and the edges are realized by disjointness. 
We briefly recall subsurface projections from \cite{MM2}. We let $R(A)$ denote a regular neighborhood of $A \subseteq S$.
\begin{itemize}
\item Non--annular projections:
suppose $Z\subseteq S$ such that $Z$ is not an essential annulus. We define a set--map $i_{Z}: C(S)\rightarrow AC(Z)$ by taking $\{x\cap Z\}$. We define a set--map $p_{Z}:AC(Z)\rightarrow C(Z)$ by taking $\{\partial(R(x\cup \partial(Z)) )\}.$ The subsurface projection to $Z$ is the composition $$\pi_{Z}=p_{Z}\circ i_{Z}:C(S)\rightarrow C(Z).$$ 

\item Annular projections:
suppose $Z\subseteq S$ such that $Z$ is an essential annulus. Take the annular cover of $S$ which corresponds to $Z$, compactify the cover with $\partial (\mathbb{H}^{2})$. We denote the resulting annular cover by $S^{Z}$. We first need to define the curve graphs for annuli. The vertices of $C(Z)$ are isotopy classes of arcs which connect two boundary components of $S^{Z}$, here the isotopy is relative to $\partial (S^{Z})$ pointwise. The edge between two vertices are realized by disjointness in the interior of $S^{Z}$. The subsurface projection to $Z$ is the set--map $$\pi_{Z}:C(S)\rightarrow C(Z)$$ by taking the lift of $x$ in $S^{Z}$. By abuse of notation, if $a\in C(S)$ then we often use notations $\pi_{a}$ and $d_{a}$ to represent $\pi_{R(a)}$ and $d_{R(a)}$ respectively.

\item For both types of projections, if $C\subseteq C(S)$ then $\pi_{Z}(C):=\cup_{c\in C} \pi_{Z}(c)$. If $C'\subseteq C(S)$ then we let $d_{Z}(C, C')$ denote $d_{Z}(\pi_{Z}(C), \pi_{Z}(C')).$ 
\end{itemize}

We note the following elementary fact from \cite{MM2}. It says, in particular, subsurface projections are coarsely well--defined.
\begin{lemma}\label{smallprojection}
Let $Z\subseteq S$. The image of a multicurve via $\pi_{Z}$ is bounded by $\3$.
\end{lemma}

\subsubsection{loxodromic actions of pure mapping classes}\label{machida}
The goal of this section is to observe Corollary \ref{gru}. First, by Theorem \ref{GD} we have
\begin{lemma}\label{min}
If $f$ is a partial pseudo--Anosov mapping class supported on $Z\subseteq S$, then there exists $k(Z)\geq \frac{1}{162 \cdot |\chi(Z)|^{2}+30 \cdot |\chi(Z)|-10n}$ such that for any $x\in C(S)$ such that $\pi_{Z}(x)\neq \emptyset$ and any $j\in \mathbb{Z},$ $$d_{Z}(x,f^{j}(x))\geq k(Z)\cdot |j|.$$ 
\end{lemma}

There are analogous results, in the setting of Dehn twists and multitwists, to the above result. 

\begin{lemma}[\cite{MM2}]\label{dehnmin}
Let $a\in C(S)$ and $t_{a}$ be Dehn twist along $a$. For any $x\in C(S)$ such that $\pi_{a}(x)\neq \emptyset$ and any $j\in \mathbb{Z}$, $$d_{a}(x,t_{a}^{j}(x))\geq |j|.$$
\end{lemma}
 
Indeed, we have $d_{a}(x,t_{a}^{j}(x))= |j|+2$ on the above lemma, see \cite{MM2}. The proof is a bit tricky because $t_{a}$ has an effect to $\pi_{a}(x)$ at every component of the lift of $a$ in $S^{R(a)}$, but it is explained in the paper of Farb--Lubotzky--Minsky \cite{FLM}; the main effect occurs at the core curve of $S^{R(a)}$ by twisting along the core curve. At the other components of the lift of $a$, which are arcs whose endpoints are contained in a single boundary component of $S^{R(a)}$, there are minor effects by sliding along on these arcs, and these effects result into sliding the endpoints of $\pi_{a}(x)$ in $\partial(S^{R(a)})$ (but not $2\pi$ much) creating $2$ extra intersections. We refer the reader to \cite[Figure 1]{FLM}. Therefore, $$i(\pi_{a}(x), \pi_{a}(t_{a}^{j}(x)) )=j+1 \Longrightarrow d_{a}(x,t_{a}^{j}(x))= |j|+2.$$ 
 
 \begin{lemma}\label{multimin}
 Let $a\in C(S)$ and $m_{a}$ be a multitwist which contains a power of $t_{a}$. For any $x\in C(S)$ such that $\pi_{a}(x)\neq \emptyset$ and any $j\in \mathbb{Z}$, $$d_{a}(x,m_{a}^{j}(x))\geq |j|.$$ 
\end{lemma}
\begin{proof}
The proof is analogous to the case of Dehn twsits. The main difference is that there are additional effects from other Dehn twists (not $t_{a}$) contained in $m_{a}$. However, the lifts of the core curves of the annuli that support other Dehn twists are arcs whose endpoints are contained in a single boundary component of $S^{R(a)}$. Therefore, those Dehn twists have minor effect to $\pi_{a}(x)$ in $S^{R(a)}$ by sliding along on the corresponding arcs where the direction is determined by the signs of the twists. As before, the position of the endpoints of $\pi_{a}(x)$ changes in $S^{R(a)}$, but not $2 \pi$ much. We have $$i(\pi_{a}(x), \pi_{a}(m_{a}^{j}(x)) )\geq j-1 \Longrightarrow d_{a}(x,m_{a}^{j}(x))\geq  |j|.$$
\end{proof}

We will use the following corollary in the proofs of Lemma \ref{sp2} and Lemma \ref{actual} in \S\ref{tl}.
\begin{corollary}\label{gru}
Let $\phi \in PMod(S)$. There exists $k(S)\geq \frac{1}{162 \cdot |\chi(S)|^{2}+30 \cdot |\chi(S)|-10n}$ such that for any $Y\in \A(\phi)$, any $x\in C(S)$ such that $\pi_{Y}(x)\neq \emptyset$, and any $j\in \mathbb{Z},$ $$d_{Y}(x,\phi^{j}(x))\geq k(S)\cdot |j|.$$
\end{corollary}
\begin{proof}
The statement follows by combining Lemma \ref{min}, Lemma \ref{dehnmin}, and Lemma \ref{multimin}. 
\end{proof}

\subsubsection{Three key tools}\label{3important}
This section collects main tools to be used in the paper. 

The following theorem is originally due to Masur--Minksy \cite{MM2}. Webb effectivized the theorem \cite{WEB1}. 
\begin{theorem}[Bounded Geodesic Image Theorem]\label{BGIT}
If every vertex of a geodesic projects nontrivially to $Z\subsetneq S$, then the image of the geodesic via $\pi_{Z}$ is bounded by $M\leq 100$.
\end{theorem}
\begin{remark}\label{rBGIT}
In this paper, we assume $M$ in Theorem \ref{BGIT} is bigger than $5$. 
\end{remark}

The following theorem is originally due to Behrstock \cite{BEH}. Leininger effectivized the theorem; a complete proof can be found in the paper of Mangahas \cite{Mangahas}. Recall that two subsurface $Z,W\subseteq S$ are said to be overlapping if they are neither disjoint nor nested. 

\begin{theorem}[Behrstock's inequality]\label{BEH}
Let $Z,W\subsetneq S$ be overlapping surfaces and $\mu$ be a multicurve. If $d_{Z}(\partial(W),\mu)>9$ then $d_{W}(\partial(Z),\mu)\leq 4.$ 
\end{theorem}

The following theorem plays an important role in the paper. It allows us to apply our technical machinery to work for almost all pure mapping classes. The theorem is originally due to Masur--Minksy \cite{MM2}. Bestvina--Bromberg--Fujiwara effectivized the theorem with an explicit constant \cite{BBF}.
\begin{theorem}[Finitely many large projections]\label{BBF}
For any given pair of curves, $x,y\in C(S)$, there are only finitely many subsurfaces $W$ with $d_{W}(x,y)>3.$ 
\end{theorem}

\section{Construction}\label{tl}
By abuse of notation, we use the following notations for the rest of the paper: let $Z,W\subseteq S$. By $\pi_{Z}(W)$, we mean $\pi_{Z}(\partial(W))$. Hence, $d_{Z}(W,)$ measures the distance from $\pi_{Z}(\partial(W))$ in $C(Z)$.  
Our construction of pseudo--Anosov mapping classes relies on the following elementary criterion: if $f\in Mod(S)$ and $||f||>0$, then $f$ is pseudo--Anosov. 

\subsection{Warm--up}\label{tl1}
The goals of this section are to prove a weaker/specific version of the main theorem in $\S\ref{tl2}$ and to construct pseudo--Anosov mapping classes which have invariant geodesics in the curve graph, which arise from the construction together with their stable lenghs.

The following proposition will capture the orbits of the mapping classes we will obtain:
\begin{proposition}\label{sp}
Let $a\in C(S)$ such that it is a non--separating curve and $g\in Mod(S)$ such that $d_{S}(a,g(a))>0.$ Let $Z$ be the complement of $g(a)$ in $S$. If $d_{Z}(a,g^{2}(a))>M+9,$ then, for all $i\geq 2$, we have 
\begin{enumerate}
\item $d_{Z}(g^{2}(a), g^{i}(a))\leq 9$.
\item $d_{S}(a,g^{i}(a))=  i\cdot d_{S}(a,g(a))$.
\end{enumerate}
\end{proposition}

\begin{proof}
We prove by the induction on $i$. 

For the base case, $i=2$, the first statement follows from Lemma \ref{smallprojection}. For the second statement, by using Theorem \ref{BGIT} with the fact that $\partial(Z)=g(a)$ is non--separating, we have
\begin{align}
   d_{S}(a,g^{2}(a))&= d_{S}(a,Z) + d_{S}(Z, g^{2}(a))\tag*{}\\
	  &= 2\cdot d_{S}(a,g(a)). \tag{Since $g(a)=\partial(Z)$}
                  \end{align}

For the inductive step, assume the statement is true for all $i\leq k$: 
we first need to check $\pi_{Z}(g^{k+1}(a))\neq \emptyset$; 
\begin{align}
d_{S}(Z,g^{k+1}(a))&= d_{S}(g(a),g^{k+1}(a)) \tag{Since $g(a)=\partial(Z)$}\\
&= d_{S}(a,g^{k}(a))\tag*{}\\
&= k\cdot d_{S}(a,g(a))\tag{By the inductive hypothesis}\\
&>1.\tag{Since $k\geq 2$ and $d_{S}(a,g(a))>0$}
   \end{align}
Therefore, $\pi_{Z}(g^{k+1}(a))\neq \emptyset$. Now, we prove, by using Theorem \ref{BEH}, that $$d_{Z}(g^{2}(a),g^{k+1}(a))\leq 9.$$ First, we note that since $d_{Z}(g^{2}(a),g^{k}(a))\leq 9$ by the inductive hypothesis, we have $d_{Z}(a,g^{k}(a))\geq d_{Z}(a,g^{2}(a))-d_{Z}(g^{2}(a),g^{k}(a)) >M.$
Now, suppose $d_{Z}(g^{2}(a),g^{k+1}(a))>9$ for a contradiction. Then
\begin{align}
d_{Z}(g^{2}(a),g^{k+1}(a) )>9   &\Longrightarrow d_{Z}(g(Z),g^{k+1}(a) )> 9 \tag{Since $g^{2}(a)=g(\partial(Z))=\partial(g(Z))$} \\
&\Longrightarrow d_{g(Z)}(Z,g^{k+1}(a) )\leq 4 \tag{by Theorem \ref{BEH} as $Z$ and $g(Z)$ overlap} \\
&\Longrightarrow d_{g(Z)}(g(a),g^{k+1}(a) )\leq 4 \tag{Since $g(a)=\partial(Z)$} \\
&\Longrightarrow d_{Z}(a,g^{k}(a) )\leq 4. \tag*{} 
\end{align}
This is a contradiction because $d_{Z}(a,g^{k}(a))>M.$ Therefore, $$d_{Z}(g^{2}(a),g^{k+1}(a))\leq 9.$$ Hence, we have $ d_{Z}(a,g^{k+1}(a)) \geq d_{Z}(a,g^{2}(a))- d_{Z}(g^{2}(a),g^{k+1}(a)) >(M+9)-9 =M,$ and we can apply Theorem \ref{BGIT} here; 
\begin{align}
   d_{S}(a,g^{k+1}(a))&= d_{S}(a,Z)+d_{S}(Z,g^{k+1}(a)) \tag{By Theorem \ref{BGIT}}\\
   &= d_{S}(a,g(a))+d_{S}(g^{}(a),g^{k+1}(a)) \tag{Since $\partial(Z)=g(a)$}\\
      	&=d_{S}(a,g(a))+ d_{S}(a,g^{k}(a)) \tag*{}\\
	&= d_{S}(a,g(a))+ k\cdot d_{S}(a,g(a))  \tag{By inductive hypothesis}\\
                  &=(k+1)\cdot d_{S}(a,g(a)). \tag*{}
                \end{align}

\end{proof}



We use Proposition \ref{sp} in the following way:

\begin{lemma}\label{sp2}
Let $a\in C(S)$ such that it is a non--separating curve and $f\in Mod(S)$ such that $d_{S}(a,f(a))>0$. Let $Y$ denote the complementary component of $a$ in $S$. Let $\phi \in \Phi_{a}$ such that $Y\in \A(\phi).$ There exists $K$ such that, for all $|j|\geq K$, $g= f\circ \phi ^{j}$ satisfy the hypothesis of Proposition \ref{sp}, i.e.
\begin{itemize}
\item $d_{S}(a,g(a))>0.$ 
\item $d_{Z}(a,g^{2}(a))>M+9$ where $Z=g(Y)$, namely $Z$ is the complement of $g(a)$ in $S$.
\end{itemize}
\end{lemma}
\begin{proof}

\begin{itemize}
\item $d_{S}(a,g(a))=d_{S}(a,f\circ \phi^{j}(a))=d_{S}(a,f(a))>0$. 

\item We have 
\begin{align}
   d_{Z}(a,g^{2}(a))&=  d_{Z}(a,( f\circ \phi^{j})^{2}(a))  \tag*{}\\
   &= d_{Z}(a, f\circ \phi^{j} \circ f(a)) \tag{Since $\phi^{j}(a)=a$}\\
   &\geq d_{Z}(f^{2}(a), f\circ \phi^{j}\circ f(a))  - d_{Z}(a,  f^{2}(a))\tag*{}\\
   &=d_{Y}(f(a), \phi^{j} \circ f(a)) - d_{Y}(f^{-1}(a),  f(a))\tag{Since $Z=g(Y)=f(Y)$}\\
  &\geq k(S)\cdot |j|- d_{Y}(f^{-1}(a),  f(a)).\tag{Since $Y \in \A(\phi)$ and Corollary \ref{gru}}
                \end{align}
Hence, take $j$ such that $$k(S)\cdot |j|- d_{Y}(f^{-1}(a),  f(a))>M+9 \Longleftrightarrow |j|>\frac{M+9+d_{Y}(f^{-1}(a),  f(a))}{k(S)}.$$ \end{itemize}

\end{proof}

We give an example of pseudo--Anosov mapping classes which can be produced by Proposition \ref{sp} and Lemma \ref{sp2}:
\begin{example}\label{spmsri}
We use the same notations from Lemma \ref{sp2}. What we need here is $f\in Mod(S)$ such that $d_{S}(a,f(a))>0$ and $d_{Y}(f^{-1}(a),  f(a))$ is small. Let $B\subsetneq S$ such that $B$ overlaps with $Y$. Let $f\in Mod(S)$ such that $$d_{B}(f^{-1}(a),a)>9 \text{ and }  d_{B}(a,f(a))>9.$$ Note that the above $f$ can be found easily by taking a high power of a pseudo--Anosov mapping class supported on $B$. (If $B$ is an annulus then we can take a high power of Dehn twist supported on $B$.) By applying Behrstock's inequality to the above inequalities, we have $$d_{Y}(f^{-1}(a),B)\leq 4 \text{ and }  d_{Y}(B,f(a))\leq 4 \Longrightarrow  d_{Y}(f^{-1}(a),f(a))\leq 8.$$
Recall $g:=f\circ \phi^{j}$. By Lemma \ref{sp2} and Proposition \ref{sp}, for all $j$ such that $|j|>\frac{M+9+8}{k(S)}$, we have 
\begin{align}
   d_{S}(a,g^{i}(a))&=  i\cdot d_{S}(a,g(a))\tag*{}\\
   &\geq i. \tag{Since $d_{S}(a,g(a))=d_{S}(a,f(a))>0$}
                \end{align}
In particular, $||g||>0$. Therefore, $g$ is pseudo--Anosov.
\end{example}




\begin{remark}\label{fix1}
Pseudo--Anosov mapping classes obtained in this section are special from the view point of their actions to $C(S)$. Bowditch showed that if $g\in Mod(S)$ is pseudo--Anosov then there exists $m$ such that $g^{m}$ fixes some geodesic in $C(S)$, i.e. $g$ fixes some quasi--geodesic \cite{BO2}. Furthermore, by the work of Webb \cite{WEB2}, $m$ is known to grow at most doubly exponentially with $\xi(S).$ Pseudo--Anosov mapping classes obtained in this section fix a geodesic in $C(S)$. Furthermore, an invariant geodesic can be found easily: let $g$ be a pseudo--Anosov mapping class obtained here. We have $d_{S}(a,g^{i}(a))=  i\cdot d_{S}(a,g(a)) \text{ for all } i$, in particualr $$||g||=d_{S}(a,g(a)) \Longrightarrow ||g||=d_{S}(a,f(a)).$$ 
Let $[x,y]$ denote a geodesic between $x,y\in C(S)$ and let $$G_{a}:=\cup_{n\in \mathbb{Z}}g^{n}([a,g(a)]).$$ Clearly $g(G_{a})=G_{a}$. Also, $G_{a}$ is a geodesic in $C(S)$ by the above equality.
\end{remark}
 
\subsection{The main theorem}\label{tl2}
We observe Proposition \ref{i3}, which is a general version of Proposition \ref{sp}. The only difference is that we require $d_{S}(a,g(a))>3$ so that  $Z$ and $g(Z)$ overlap, see Proposition \ref{pi3}. This is a technical condition so that we can use Behrstock's inequality.

\begin{proposition}\label{pi3}
Let $a\in C(S)$ and $g\in Mod(S)$ such that $d_{S}(a,g(a))>3.$ Let $Z\subsetneq S$ such that $d_{S}(Z,g(a))\leq 1$. We have 
\begin{enumerate}
\item $Z$ and $g(Z)$ overlap. 
\item $d_{g(Z)}(g(a), Z)\leq \3.$
\item $d_{Z}(g^{2}(a), g(Z))\leq \3.$

\end{enumerate}
\end{proposition}
\begin{proof}
\begin{enumerate} 
\item $d_{S}(Z, g(Z))\geq d_{S}(g(a),g^{2}(a))- d_{S}(Z, g(a))- d_{S}(g(Z), g^{2}(a)) \geq d_{S}(g(a),g^{2}(a))-2>1$.
Therefore, $Z$ and $g(Z)$ overlap.

\item First, we check that 
\begin{itemize}
\item $\pi_{g(Z)}(g(a)) \neq \emptyset$ since $d_{S}(g(Z), g(a))\geq d_{S}(g^{2}(a), g(a)) -1>1$. 
\item $\pi_{g(Z)}(Z)\neq \emptyset$ by the first statement. 
\end{itemize}
By Lemma \ref{smallprojection}, we have $d_{g(Z)}(g(a), Z)\leq \3$ since $d_{S}(g(a),Z)\leq1$. 

\item An analogous argument to the above applies here. We only check that $\pi_{Z}(g^{2}(a)) \neq \emptyset$ since $d_{S}(Z, g^{2}(a))\geq d_{S}(g(a), g^{2}(a)) -1>1$. 
\end{enumerate}
\end{proof}

The proof of the following proposition is identical to that of Proposition \ref{sp}, but for completeness, we briefly go over.

\begin{proposition}\label{i3}
Let $a\in C(S)$ and $g\in Mod(S)$ such that $d_{S}(a,g(a))>3.$ Let $Z\subsetneq S$ such that $d_{S}(Z,g(a))\leq 1$. If $d_{Z}(a,g^{2}(a))>M+11,$ then, for all $i\geq 2$, we have 
\begin{enumerate}
\item $d_{Z}(g^{2}(a), g^{i}(a))\leq 11$.
\item $d_{S}(a,g^{i}(a))\geq  i\cdot (d_{S}(a,g(a))-4)+4$.
\end{enumerate}
\end{proposition}
\begin{proof}

We prove by the induction on $i$.

For the base case, $i=2$, the first statement follows from Lemma \ref{smallprojection}. For the second statement, we observe 
\begin{align}
   d_{S}(a,g^{2}(a))
   	&\geq (d_{S}(a,g(a))-2) + (d_{S}(g(a), g^{2}(a))-2) \tag{By Theorem \ref{BGIT} and $d_{S}(Z,g(a))\leq 1$}\\
	&=2\cdot (d_{S}(a,g(a))-4)+4.  \tag*{}
                  \end{align}

For the inductive step, assume the statement is true for all $i\leq k$: 
we first need to check $\pi_{Z}(g^{k+1}(a))\neq \emptyset$; 
\begin{align}
d_{S}(Z,g^{k+1}(a))
&\geq d_{S}(g(a),g^{k+1}(a))-  1\tag{Since $d_{S}(g(a),Z)\leq 1$}\\
&\geq k\cdot (d_{S}(a,g(a))-4)+4 -1\tag{By inductive hypothesis}\\
&>1.\tag{Since $d_{S}(a,g(a))>3$}
   \end{align}
Therefore, $\pi_{Z}(g^{k+1}(a))\neq \emptyset$. Now, we prove $$d_{Z}(g^{2}(a),g^{k+1}(a))\leq11.$$
First, since $d_{Z}(g^{2}(a),g^{k}(a))\leq11$ by the inductive hypothesis, we have $d_{Z}(a,g^{k}(a))>M.$
Now, suppose $d_{Z}(g^{2}(a),g^{k+1}(a))>11$ for a contradiction. Then 
\begin{align}
 d_{Z}(g^{2}(a),g^{k+1}(a) )> 11  &\Longrightarrow d_{Z}(g(Z),g^{k+1}(a) )> 9 \tag{by Proposition \ref{pi3}} \\
   &\Longrightarrow d_{g(Z)}(Z,g^{k+1}(a) )\leq 4 \tag{by Theorem \ref{BEH} and Proposition \ref{pi3}} \\
      &\Longrightarrow d_{g(Z)}(g(a),g^{k+1}(a) )\leq 6 \tag{by Proposition \ref{pi3}} \\
      &\Longrightarrow d_{Z}(a,g^{k}(a) )\leq 6.& \tag*{} 
\end{align}
This is a contradiction because $d_{Z}(a,g^{k}(a))>M.$ See Remark \ref{rBGIT}. Therefore, $$d_{Z}(g^{2}(a),g^{k+1}(a))\leq 11.$$ Hence, we have $d_{Z}(a,g^{k+1}(a))\geq d_{Z}(a,g^{2}(a))- d_{Z}(g^{2}(a),g^{k+1}(a))>M.$
By Theorem \ref{BGIT}, we have
\begin{align}
   d_{S}(a,g^{k+1}(a))
   &\geq d_{S}(a,g(a))+d_{S}(g^{}(a),g^{k+1}(a))-4 \tag*{}\\
	&\geq d_{S}(a,g(a))+ k\cdot (d_{S}(a,g(a))-4)+4-4 \tag*{}\\
                  &=(k+1)\cdot (d_{S}(a,g(a))-4)+4. \tag*{}
                \end{align}
\end{proof}

      
\begin{lemma}\label{actual}
Let $a\in C(S)$ and $f\in Mod(S)$ such that $d_{S}(a,f(a))>3$. Let $\phi \in \Phi_{a}$. There exists $K$ such that, for all $|j|\geq K$, $g= f\circ \phi ^{j}$ satisfy the hypothesis of Proposition \ref{i3}, i.e.
\begin{itemize}
\item $d_{S}(a,g(a))>3.$ 
\item $d_{Z}(a,g^{2}(a))>M+11$ where $Z=g(Y)$ where $Y\in  \A(\phi) $. Note that $d_{S}(Z,g(a))=d_{S}(Y,a)\leq 1$ since $\A(\phi)\subseteq \mathcal{Z}_{a}$ as $\phi \in \Phi_{a}$.
\end{itemize}

\end{lemma}

\begin{proof}
\begin{itemize}
\item $d_{S}(a,g(a))=d_{S}(a,f\circ \phi^{j}(a))=d_{S}(a,f(a))>3$. 

\item First we note $Z=g(Y)=f\circ \phi^{j}(Y)=f(Y)$. We have 
\begin{align}
   d_{Z}(a,g^{2}(a))
   &= d_{Z}(a, f\circ \phi^{j} \circ f(a)) \tag{Since $\phi^{j}(a)=a$}\\
   &\geq d_{Z}(f^{2}(a), f\circ \phi^{j}\circ f(a))  - d_{Z}(a,  f^{2}(a))\tag*{}\\
   &=d_{Y}(f(a), \phi^{j} \circ f(a)) - d_{Y}(f^{-1}(a),  f(a))\tag{Since $Z=f(Y)$}\\
  &\geq k(S)\cdot |j|- d_{Y}(f^{-1}(a),  f(a)).\tag{By $Y \in \A(\phi)$ and Corollary \ref{gru}}
                \end{align}
Hence, if $j$ is such that $$k(S)\cdot |j|- d_{Y}(f^{-1}(a),  f(a))>M+11 \Longleftrightarrow |j|>\frac{M+11+d_{Y}(f^{-1}(a),  f(a))}{k(S)},$$ then $g=f\circ \phi^{j}$ satisfies the hypothesis of Proposition \ref{i3}.
\end{itemize}
\end{proof}

We have the main theorem of this paper. The proof follows by combining Proposition \ref{i3} and Lemma \ref{actual} with Theorem \ref{BBF}.

\begin{theorem}\label{whatweprove}
Let $f\in Mod(S)$ and $a \in C(S)$ such that $d_{S}(a,f(a))>4$. There exists $\mathcal{W} \subseteq \mathcal{Z}_{a}$ with $|\mathcal{W}|<\infty$ such that the following holds: if $\phi \in \Phi_{a}$ such that $\A(\phi) \cap \{\mathcal{Z}_{a} \setminus \mathcal{W} \}\neq \emptyset$, then $f\circ \phi^{j}$ is pseudo--Anosov whenever $|j|> \frac{M+14}{k(S)}$.
\end{theorem}
\begin{proof}
Let $\mathcal{W}:=\{W\subseteq \mathcal{Z}_{a}| d_{W}(f^{-1}(a),f(a))>3\}.$ By Theorem \ref{BBF}, we have $|\mathcal{W}|<\infty.$ Let $g=f\circ \phi^{j}.$ Since $\phi \in \Phi_{a}$ such that $\A(\phi) \cap \{\mathcal{Z}_{a} \setminus \mathcal{W} \}\neq \emptyset$, by Lemma \ref{actual} and Proposition \ref{i3}, for all $j$ such that $|j|>\frac{M+11+3}{k(S)}$, we have 
\begin{align}
   d_{S}(a,g^{i}(a))&\geq  i\cdot (d_{S}(a,g(a))-4)+4 \tag*{}\\
   &\geq i+4. \tag{Since $d_{S}(a,g(a))=d_{S}(a,f(a))>4$}
                \end{align}
In particular, $||g||>0$. Therefore, $g=f\circ \phi^{j}$ is pseudo--Anosov.
\end{proof}

\begin{remark}
We note that Theorem \ref{whatweprove} holds for an opposite family, $\{\phi^{j} \circ f | j\in \mathbb{Z}\}$: let $\langle h \rangle\cdot a$ denote the orbit of $a$ under the iteration of $h\in Mod(S)$ in $C(S)$. We observe that $$\langle \phi^{j}\circ f \rangle \cdot a=\phi^{j}( \langle  f\circ \phi^{j}  \rangle\cdot a).$$
\end{remark}

By the arguments given in this section, we have

\begin{corollary}\label{cccc1}
Let $f\in Mod(S)$ and $a \in C(S)$ such that $d_{S}(a,f(a))>4$. If $\phi \in \Phi_{a}$ such that $i(f^{}(a), \partial(Y)) >2\cdot i(f^{-1}(a), f(a))$ or $i(f^{-1}(a), \partial(Y))>2\cdot i(f^{-1}(a), f(a))$ for some $Y\in \A(\phi)$, then $\phi^{j}\circ  f$ is pseudo--Anosov whenever $|j|> \frac{M+14}{k(S)}$.
\end{corollary}
\begin{proof}
By the proof of Theorem \ref{BBF}, which we refer the reader to \cite{BBF}, if $W\in \mathcal{W}$ then we must have $$i(f^{}(a), \partial(W)) \leq 2\cdot i(f^{-1}(a), f(a)) \text{ and }i(f^{-1}(a), \partial(W))\leq 2\cdot i(f^{-1}(a), f(a)).$$ Therefore, if $$i(f^{}(a), \partial(Y)) >2\cdot i(f^{-1}(a), f(a)) \text{ or }i(f^{-1}(a), \partial(Y))>2\cdot i(f^{-1}(a), f(a))$$ then we must have $d_{Y}(f^{-1}(a),f(a))\leq3$. Hence, $\phi^{j}\circ  f$ is pseudo--Anosov whenever $|j|> \frac{M+14}{k(S)}$. 
\end{proof}

\begin{remark}\label{dry1}
A pure mapping class $\phi$ in Corollary \ref{cccc1} can be found easily. For instance, one can pick $\gamma \in C(S)$ such that $i(a,\gamma)=0$ and $i(f(a),\gamma)>2\cdot i(f^{-1}(a), f(a))$, and take $\phi$ to be Dehn twist along $\gamma$.
\end{remark}

We also have

\begin{corollary}\label{cccc2}
Let $f\in Mod(S)$ and $a \in C(S)$ such that $d_{S}(a,f(a))>4$ and that there exists a geodesic between $f^{-1}(a)$ and $f(a)$ whose $1$--neighborhood does not intersect with $\{\partial(Z)|Z\in \mathcal{Z}_{a}\}$ in $C(S)$. If $\phi \in \Phi_{a}$, then $\phi^{j}\circ  f$ is pseudo--Anosov whenever $|j|> \frac{2\cdot M+11}{k(S)}.$ 
\end{corollary}
\begin{proof}
Since there exists a geodesic between $f^{-1}(a)$ and $f(a)$ whose $1$--neighborhood does not intersect with $\{\partial(Z)|Z\in \mathcal{Z}_{a}\}$ in $C(S)$, we have that $d_{Z}(f^{-1}(a),f(a))\leq M$ for all $Z\in \mathcal{Z}_{a}$ by Theorem \ref{BGIT}. Because $\A(\phi)\subseteq \mathcal{Z}_{a}$, $\phi^{j}\circ  f$ is pseudo--Anosov whenever $|j|> \frac{2\cdot M+11}{k(S)}$. 
\end{proof}
 
 \begin{remark}\label{dry2}
A pair $f\in Mod(S)$ and $a \in C(S)$ in Corollary \ref{cccc2} can be found easily. For instance, one can first pick a non--separating curve $\gamma \in C(S)$ such that $d_{S}(a,\gamma)\gg 0$ and take a partial pseudo--Anosov map $g$ supported on the complement of $\gamma$. There exists $p$ such that $d_{S}(g^{-p}(a),g^{p}(a))=2\cdot d_{S}(g^{p}(a),\gamma)$ by Lemma \ref{min} and Theorem \ref{BGIT}, which means $\gamma$ is the midpoint of a geodesic between $g^{-p}(a)$ and $g^{p}(a)$. Take $f=g^{p}$. Recall that the curve graph is hyperbolic (Originally due to Masur--Minsky \cite{MM1}, also worked out by Aougab \cite{AOU}, Bowditch \cite{BO1}\cite{BO4}, Hamenst{\"a}dt \cite{HAM}, Clay--Rafi--Schleimer \cite{CRS}, and Hensel--Przytycki--Webb \cite{HPW}); since $d_{S}(f^{-1}(a),a)= d_{S}(a,f(a))$, $d_{S}(a,\gamma)\gg 0$, and $d_{S}(\partial(Z),a)\leq 1$ for all $Z\in \mathcal{Z}_{a}$, we have the added assumption. 
Finally, we observe that $d_{S}(a,f(a))> 4$, otherwise $d_{S}(a,\gamma)$ would be small.
\end{remark}

Lastly, we remark on the stable lengths of pseudo--Anosov mapping class obtained in this section:

\begin{remark}\label{fix2}
As in Remark \ref{fix1}, if $g$ is a pseudo--Anosov mapping class obtained in this section, then we have $$d_{S}(a,g(a))-4 \leq ||g||\leq d_{S}(a,g(a))\Longrightarrow d_{S}(a,f(a))-4 \leq ||g||\leq d_{S}(a,f(a)).$$ 
\end{remark}

\bibliographystyle{plain}
\bibliography{references.bib}

\end{document}